\documentclass[runningheads]{llncs}
\usepackage{amssymb}
\usepackage{graphicx}
\usepackage{url}
\usepackage{amsmath, mathrsfs, multirow, subfigure}
\usepackage{ifthen} 
\usepackage{amsfonts}
\usepackage[usenames]{color}
\usepackage[normalem]{ulem}
\usepackage{tcolorbox}
\usepackage{soul}
\usepackage[ruled,vlined]{algorithm2e}
\usepackage{float}

\newcommand{\keywords}[1]{\par\addvspace\baselineskip
\noindent\keywordname\enspace\ignorespaces#1}

\usepackage[sort&compress]{natbib}
\bibpunct{(}{)}{;}{a}{}{,} 
\makeatletter
\renewcommand\bibsection%
{
  \section*{Bibliography}
  \section*{\refname
    \@mkboth{\MakeUppercase{\refname}}{\MakeUppercase{\refname}}}
}
\makeatother

\newcommand{\prob}{\mathsf{P}}

\newcommand{\risk}{\mathsf{R}}

\newcommand{\rank}{\text{rank}}

\def\cR{\mathcal{R}}
\def\cY{\mathcal{Y}}
\def\rank{\mathrm{rank}}

\DeclareMathOperator*{\argmin}{arg\,min}

\newcommand{\cred}{\mathscr{C}}

\newcommand{\umu}{\overline{\mu}}

\begin{document}
\title{Decision theory via model-free generalized fiducial inference}
\titlerunning{Decision theory via MFGF inference}
\author{Jonathan P Williams\inst{1}\and
Yang Liu\inst{2}}
\authorrunning{J~P~Williams and Y~Liu}
\institute{Department of Statistics \\ North Carolina State University, Raleigh, North Carolina 27695, US \\
\email{jwilli27@ncsu.edu} \and
Department of Human Development and Quantitative Methodology \\ University of Maryland, College Park, Maryland 20742, US \\
\email{yliu87@umd.edu}
}
\maketitle              
\begin{abstract}
Building on the recent development of the model-free generalized fiducial (MFGF) paradigm \citep{williams2023} for predictive inference with finite-sample frequentist validity guarantees, in this paper, we develop an MFGF-based approach to decision theory.  Beyond the utility of the new tools we contribute to the field of decision theory, our work establishes a formal connection between decision theories from the perspectives of fiducial inference, conformal prediction, and imprecise probability theory.  In our paper, we establish pointwise and uniform consistency of an {\em MFGF upper risk function} as an approximation to the true risk function via the derivation of nonasymptotic concentration bounds, and our work serves as the foundation for future investigations of the properties of the MFGF upper risk from the perspective of new decision-theoretic, finite-sample validity criterion, as in \cite{martin2021}.
\keywords{Choquet integral; empirical risk; possibility theory; upper prevision; upper probability.}
\end{abstract}

\section{Introduction}\label{s.intro}

Decision theory is an important topic in statistical inference, where the goal is to determine an optimal decision rule based on observed data.  Such problems are typically mathematically specified via minimizing an expected value of a loss function.  The choice of loss function is application and context specific, and is considered to be given by the practitioner.  Our goal in this article is to develop a decision theory based on imprecise probabilistic inference inherited from the model-free generalized fiducial (MFGF) inference paradigm.  

The MFGF paradigm was introduced in \cite{williams2023} as a mechanism for building model-free predictive distributions, based on imprecise probability theory, that both, serve as general purpose inferential tools and yield prediction sets with finite-sample, frequentist coverage guarantees.  The MFGF paradigm establishes a formal connection between three important disciplines focused on the development of uncertainty quantification; namely, fiducial inference \citep{fisher1935,hannig2016}, conformal prediction \citep{vovk2022}, and imprecise probability theory \citep{walley1991,augustin2014}.  Thus, the contribution of our paper is the development of a new approach to decision theory, that also establishes a formal connection among the three mentioned disciples on the topic of decision theory.  Fiducial inference for decision theory has previously been considered in \cite{taraldsen2024}, and references therein, based on [precise] probability calculus; in the context of conformal prediction, a few papers now exist under the umbrella of ``conformal risk control'' \citep{angelopoulos2023}; and imprecise probabilistic developments in decision theory are surveyed in \cite{huntley2014} and \cite{denoeux2019}.  Additionally relevant to our work are the decision-theoretic developments in the inferential models context in \cite{martin2021}.

\section{Background on MFGF}

Assume the variables $Y_{1},\dots,Y_{n+1} \overset{\text{iid}}{\sim} Y$ are continuous and $\cY$-valued, and let $\Psi : \cY^{n}\times\cY \to \cR$ be a nonconformity measure (i.e., a measurable function that is invariant to the order of its first $n$ components) of interest.  In this case, the random variables $t_{i}(Y_{i}) := \Psi(Y^{n+1}_{-i},Y_{i})$, for $i \in \{1,\dots,n+1\}$, are exchangeable, in which $Y_{-i}^{n+1} := (Y_1,\dots,Y_{i-1},Y_{i+1},\dots,Y_{n+1})$, and so the position or {\em rank} of $t_{i}(Y_{i})$ in the order statistics (in ascending order) of the sample $t_{1}(Y_{1}),\dots,t_{n+1}(Y_{n+1})$ follows a discrete uniform distribution over $\{1,\dots, n+1\}$.  

From a generalized fiducial (GF) inference perspective this rank statistic represents a pivot that can be used to quantify uncertainty in $Y_{n+1}$ via a data generating {\em association}:
\begin{equation}\label{model_free_dge}
\rank\{t_{n+1}(Y_{n+1})\} = V \sim \text{uniform}\{1,\dots, n+1\}.
\end{equation}
This data generating {\em association} is agnostic to the choice of model for $Y$, and is investigated in \cite{williams2023} in the context of MFGF predictive inference for $Y_{n+1}$ given a sample $Y_{1},\dots,Y_{n}$.  The GF inference algorithm as in \cite{hannig2016} is then applied by first replacing the unobserved true auxiliary variable in (\ref{model_free_dge}) with an independent copy, $V^{\star} \sim \text{uniform}\{1,\dots, n+1\}$.  

Next, the GF {\em switching principle} is applied to obtain an imprecise GF distribution of $Y_{n+1}$ as a distribution over the random {\em focal sets},
\begin{equation}\label{imprecise_gf}
A_{n}(V^{\star}) := \argmin_{y\in\cY}\big\{|\rank[t_{n+1}(y)] - V^{\star}|\big\} = \big\{ y \, : \, \rank[t_{n+1}(y)] = V^{\star} \big\},
\end{equation}
where with respect to the imprecise GF probability measure denoted by $\mu : 2^\cY \to [0,1]$, for each $k \in \{1,\dots,n+1\}$,
\begin{equation}\label{eq:focal_pt_prob}
\mu\{A_{n}(V^{\star}) = A_{n}(k)\} = \frac{1}{n+1},
\end{equation}
by virtue of the discrete uniform probability mass associated with the auxiliary variable $V^{\star}$. 

As developed in \cite{williams2023}, predictive inference for $Y_{n+1}$ can be carried out in an imprecise fashion via the construction of lower/upper probability measures (such as belief/plausibility functions) or in a more typical [precise] fashion via a choice of some {\em compatible} probability measure, i.e., contained in the {\em credal set} of $\mu$,
$
\cred(\mu) := \big\{\text{probability measures } \Delta \, : \, \Delta(B) \le \umu(B), \text{ for any measurable set } B \big\}.
$

As developed in \cite{williams2023}, prediction sets that achieve at least $1-\alpha$ frequentist coverage, for any finite sample size $n$, can be constructed on $\cY$, for any level $1-\alpha \in [0,1]$, as the union $\Omega_{n}(k) := \cup_{1\le v \le k}A_{n}(v)$, where $k$ is the smallest integer such that $k \ge (1-\alpha)(n+1)$:
\[
\prob\big\{ Y_{n+1} \in \Omega_{n}(k)\big\} = \prob\big[\rank\{t_{n+1}(Y_{n+1})\} \le k\big] = \frac{k}{n+1} \ge 1-\alpha,
\]
where the second equality follows by the exchangeability of $t_{1}(Y_{1}),\dots,t_{n+1}(Y_{n+1})$.  Alternatively, these prediction sets can be defined as the upper-level sets of the contour function given in Definition \ref{def:gf_transducer}:
\begin{definition}\label{def:gf_transducer}
A GF transducer function is a contour function $f_{n} : \cY \to \cR^{+}$ defined by $f_{n}(y) := \mu\{\Omega_{n}(V^{\star})\ni y\}$.  See \cite{williams2023} for more details.
\end{definition}
See Figure~\ref{fig:loss} for graphical illustrations of two GF transducers based on a set of $n = 4$ hypothetical observed data points.

Lastly, an important property of GF probabilities is that they are frequentist probabilities on the focal sets, i.e., for every $k \in \{1,\dots,n+1\}$,
\[
\prob\big\{ Y_{n+1} \in A_{n}(k)\big\} = \frac{1}{n+1} = \mu\big\{ A_{n}(V^{\star}) = A_{n}(k)\big\},
\]
and by Lemma 8 in \cite{williams2023} $\Delta \in \cred(\mu)$ if and only if $\forall k \in \{1,\dots,n+1\}$,
\begin{equation}\label{eq:precise_prob}
\Delta\{A_{n}(k)\} = \frac{1}{n+1} = \mu\{A_{n}(k)\} = \prob\{A_{n}(k)\}.
\end{equation}

\section{Decision theory with imprecision}

Given a loss function $\ell : \Theta\times\cY \to \cR^{+}$ and a data generating model $Y \sim \prob$, traditional decision theory \citep[e.g.,][]{shalev2014,dey2023} begins by defining the true risk function 
\begin{equation}\label{eq:true_risk}
\risk(\vartheta) := \int \ell(\vartheta,y) \, d\prob(y),
\end{equation}
which is approximated by the empirical risk, $\risk_{n}(\vartheta) := n^{-1}\sum_{i=1}^{n}\ell(\vartheta,y_{i})$.  Alternatively, from the above MFGF developments, we could approximate the true risk by 
$
\risk_{\Delta}(\vartheta) := \int \ell(\vartheta,y) \, d\Delta(y),
$
for any $\Delta \in \cred(\mu)$.  Any probability measure $\Delta \in \cred(\mu)$, i.e., satisfying (\ref{eq:precise_prob}), can be regarded as a precise-probabilistic approximation to the imprecise MFGF distribution, and is guaranteed to produce prediction sets that achieve at least $1-\alpha$ frequentist coverage, for any finite sample size $n$, at any level $1-\alpha \in [0,1]$. 

Every probability measure $\Delta \in \cred(\mu)$ is, however, similarly compatible with the observed data, and so relying on the risk associated with any single $\Delta \in \cred(\mu)$ is not justified without further assumptions on the data generating mechanism.  Moreover, the false confidence theorem \citep{carmichael2018,martin2019} makes it clear that statistical inference based on any probability measure $\Delta$ is provably unreliable.  Alternatively, as advocated in \cite{martin2021}, we can define an imprecise analogue of the risk via an upper prevision/expectation as
\begin{equation}\label{eq:upper_prev_general}
\risk_{\mu}(\vartheta) := \sup\{\risk_{\Delta}(\vartheta) \, : \, \Delta \in \cred(\mu)\}.
\end{equation}
This construction makes it possible to consider a measure of risk that reflects all probability measures in the credal set $\cred(\mu)$, and as discussed in \cite{martin2021}, this imprecision plays a fundamental role in the construction of decision-theoretic, finite-sample validity criterion.  In what follows here, however, we investigate consistency properties of this MFGF upper risk in simple settings to facilitate our intuitions and to motivate a future, extended version of this manuscript where we endeavor to establish decision-theoretic, finite-sample validity properties.

\subsection{MFGF upper risk function}

Most importantly, for the MFGF upper risk to be reliable, it must approximate in some sense the true risk function as defined in equation \eqref{eq:true_risk}, i.e., an expectation with respect to the true but unknown distribution of the data, $\prob$.  A standard notion for such an approximation is via the construction of a nonasymptotic concentration bound to precisely determine a probabilistic upper bound on the rate at which the MFGF upper risk will concentrate in a neighborhood of the true risk, if at all.  As we will argue in establishing Theorems \ref{thm:consistency_pointwise} and \ref{thm:consistency_uniform}, the MFGF upper risk, is up to an error on the order of $1/n$, equivalent to the empirical risk for bounded data.  Based on this mathematical reasoning, it can be established that the MFGF upper risk will approximate the true risk in settings where the empirical risk approximates the true risk, e.g., when the uniform convergence property holds, as in Definition \ref{def:ucp}.

As it turns out, the MFGF upper risk in \eqref{eq:upper_prev_general} reduces to
\begin{equation}\label{eq:upper_prev}
\risk_{\mu}(\vartheta) = \frac{1}{n+1}\sum_{v=1}^{n+1}\sup_{y\in A_{n}(v)}\{\ell(\vartheta,y)\}.
\end{equation}
To facilitate an explicit expression for \eqref{eq:upper_prev} we take the nonconformity score as the identify function, i.e., $t_{i}(Y_{i}) = Y_{i}$, so that $A_{n}(v) = ( Y_{(v-1)}, Y_{(v)})$ for $v \in \{1,\dots,n+1\}$, where $Y_{(0)} := a$ and $Y_{(n+1)} := b$.  Recalling that $\mu\{A_{n}(v)\} = \frac{1}{n+1}$, this special setup of the MFGF distribution is equivalent to the {\em Dempster-Hill procedure} \citep{hill1968,coolen1998,vovk2024} for bounded data.  Finally, we assume that $\ell( \vartheta, \cdot) : [a,b] \to \cR^{+}$ is a convex loss function for every $\vartheta\in\Theta$, meaning that it is also convex on $A_{n}(v)$ for every $v \in \{1,\dots,n+1\}$.  See the left panel of Figure~\ref{fig:loss} for a simple graphical illustration of this setup based on $n = 4$ hypothetical observed data points. 

\begin{figure}[!t]
  \centering
  \includegraphics[width=\textwidth]{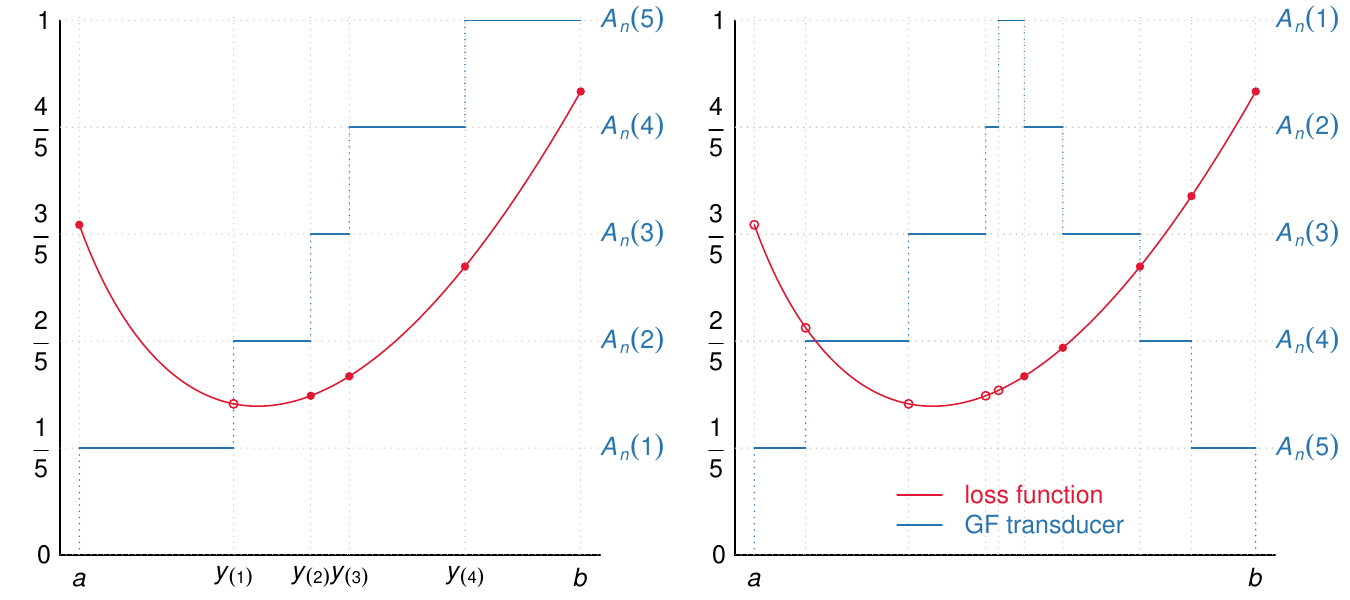}
  \caption{\footnotesize The left panel corresponds to $\Psi_i(Y_{-i}^{n + 1}, Y_i) = Y_i$ (i.e., a Dempster-Hill procedure), and the right panel corresponds to $\Psi_i(Y_{-i}^{n + 1}, Y_i) = |Y_i - \mathrm{mean}(Y_{-i}^{n + 1})|$.  The suprema of the loss function in equation \eqref{eq:upper_prev} must occur at boundaries of the focal sets, leading to the expression given in \eqref{eq:upper_prev_intermediate}; we highlight those suprema as filled dots, whose average amounts to the MFGF upper risk.}
  \label{fig:loss}
\end{figure}

In this case, 
\begin{equation}\label{eq:upper_prev_intermediate}
\sup_{y\in A_{n}(v)}\{\ell(\vartheta,y)\} = \max\{\ell(\vartheta,Y_{(v-1)}), \ell(\vartheta,Y_{(v)})\},
\end{equation}
in such a manner that 
\begin{align}\label{eq:upper_prev_simplified}
\risk_{\mu}(\vartheta) 
& = \frac{1}{n+1}\Bigg[\bigg\{\sum_{v=0}^{n+1}\ell(\vartheta,Y_{v})\bigg\} - \min_{0\le j\le n+1}\{\ell(\vartheta,Y_{j})\}\Bigg] \nonumber \\
& = \frac{1}{n+1}\bigg[n\risk_{n}(\vartheta) + \underbrace{\ell(\vartheta,a) + \ell(\vartheta,b) - \min_{0\le j\le n+1}\{\ell(\vartheta,Y_{j})\}}_{\displaystyle =: M(\vartheta)}\bigg].
\end{align}


\begin{theorem}[pointwise consistency of the MFGF upper risk]\label{thm:consistency_pointwise}
Assuming $\ell( \vartheta, \cdot) : [a,b] \to \cR^{+}$ is a convex loss function for every $\vartheta \in \Theta$, for every $\epsilon > 0$ and for all $n \ge 3M/\epsilon - 1$,
\[
\prob\big\{|\risk_{\mu}(\vartheta) - \risk(\vartheta)| > \epsilon\big\} \le 2e^{-\frac{2}{9}n\epsilon^2/L^2(\vartheta)},
\]
where $\displaystyle M := \sup_{\vartheta\in\Theta}\ell(\vartheta,a) + \sup_{\vartheta\in\Theta}\ell(\vartheta,b)$ and $\displaystyle L(\vartheta) := \sup_{y\in[a,b]}\ell(\vartheta,y) - \inf_{y\in[a,b]}\ell(\vartheta,y)$.
\end{theorem}

\begin{proof}
By the triangle inequality,
\[
|\risk_{\mu}(\vartheta) - \risk(\vartheta)| \le |\risk_{\mu}(\vartheta) - n\risk_{n}(\vartheta)/(n+1)| + \risk_{n}(\vartheta)/(n+1) + |\risk_{n}(\vartheta) - \risk(\vartheta)|,
\]
and we construct probabilistic bounds separately for each of the three terms on the right side of the inequality.

First observe that $M(\vartheta)$, as defined in \eqref{eq:upper_prev_simplified}, is nonnegative:
\begin{align*}
\min_{0\le j\le n+1}\{\ell(\vartheta,Y_{j})\} & = \min\{\ell(\vartheta,a), \ell(\vartheta,Y_{1}), \dots, \ell(\vartheta,Y_{n}), \ell(\vartheta,b)\} \\
& \le \min\{\ell(\vartheta,a), \ell(\vartheta,b)\},
\end{align*}
and so
\begin{equation}\label{eq:M_bound}
\begin{split}
M(\vartheta) & = \max\{\ell( \vartheta, a), \ell( \vartheta, b)\} + \min\{\ell( \vartheta, a), \ell( \vartheta, b)\} - \min_{0\le j\le n+1}\{\ell(\vartheta,Y_{j})\} \\
& \ge \max\{\ell( \vartheta, a), \ell( \vartheta, b)\}. 
\end{split}
\end{equation}
Then 
\[
0 \le \risk_{\mu}(\vartheta) - \frac{n\risk_{n}(\vartheta)}{n+1} = \frac{M(\vartheta)}{n+1} \le \frac{\ell(\vartheta,a) + \ell(\vartheta,b)}{n+1} \le \frac{M}{n+1}.
\]
Thus, by equation \eqref{eq:upper_prev_simplified}, 
\begin{equation}\label{eq:term_1}
\prob\big\{|\risk_{\mu}(\vartheta) - n\risk_{n}(\vartheta)/(n+1)| > \epsilon/3\big\} \le 1\big\{M/(n+1) > \epsilon/3\big\}.
\end{equation}
Next, by the convexity of $\ell( \vartheta, \cdot)$ over the closed and bounded interval $[a,b]$,
\begin{equation}\label{eq:term_2}
\begin{split}
\prob\big\{\risk_{n}(\vartheta)/(n+1) > \epsilon/3\big\} & \le 1\big[\max\{\ell( \vartheta, a), \ell( \vartheta, b)\}/(n+1) > \epsilon/3\big] \\
& \le 1\big\{M/(n+1) > \epsilon/3\big\}.
\end{split}
\end{equation}
where the second inequality is due to equation \eqref{eq:M_bound}.  Finally, considering the bounds $\inf_{y\in[a,b]}\ell(\vartheta,y) \le \ell(\vartheta,Y_{i}) \le \sup_{y\in[a,b]}\ell(\vartheta,y)$, from Hoeffding's inequality it follows that
\begin{equation}\label{eq:term_3}
\prob\big\{|\risk_{n}(\vartheta) - \risk(\vartheta)| > \epsilon/3\big\} \le 2e^{-\frac{2}{9}n\epsilon^2/L^2(\vartheta)}.
\end{equation}

The result of theorem is established by combining bounds \eqref{eq:term_1}, \eqref{eq:term_2}, and \eqref{eq:term_3} via the triangle inequality given at the beginning of the proof. $\hfill\square$
\end{proof}

\begin{definition}[uniform convergence property]\label{def:ucp}
A data generating distribution $\prob$ along with a loss function $\ell : \Theta\times\cY \to \cR^{+}$ is said to have the uniform convergence property if there exists a function $g : \cR^{+}\times\cR^{+} \to \cR$ such that for any $\epsilon > 0$ and $\alpha > 0$, if $n \ge g(\epsilon,\alpha)$, then 
\[
\prob\Big\{\sup_{\vartheta}|\risk_{n}(\vartheta) - \risk(\vartheta)| > \epsilon\Big\} < \alpha.
\]
The function $g$ is referred to as the witness function.  See, e.g., \cite{dey2023} for more details.
\end{definition}

\begin{theorem}[uniform consistency of the MFGF upper risk]\label{thm:consistency_uniform}
Assume that $\ell( \vartheta, \cdot) : [a,b] \to \cR^{+}$ is a convex loss function for every $\vartheta \in \Theta$, and that the uniform convergenve property holds, as in Definition \ref{def:ucp}.  Then for every $\epsilon > 0$, for every $\alpha > 0$, and for all $n \ge \max\{ g(\epsilon,\alpha), 3M/\epsilon - 1\}$,
\[
\prob\Big\{\sup_{\vartheta}|\risk_{\mu}(\vartheta) - \risk(\vartheta)| > \epsilon\Big\} < \alpha.
\]
\end{theorem}

\begin{proof}
Following the proof of Theorem \ref{thm:consistency_pointwise}, observe that the probabilistic bounds in \eqref{eq:term_1} and \eqref{eq:term_2} actually hold uniformly over $\vartheta\in\Theta$. 
Then an application of the triangle inequality given at the beginning of the proof of Theorem \ref{thm:consistency_pointwise} completes the proof, by the assumption of the uniform convergence property. $\hfill\square$
\end{proof}


\section{Numerical illustration}
\begin{figure}[!t]
  \centering
  \includegraphics[width=\textwidth]{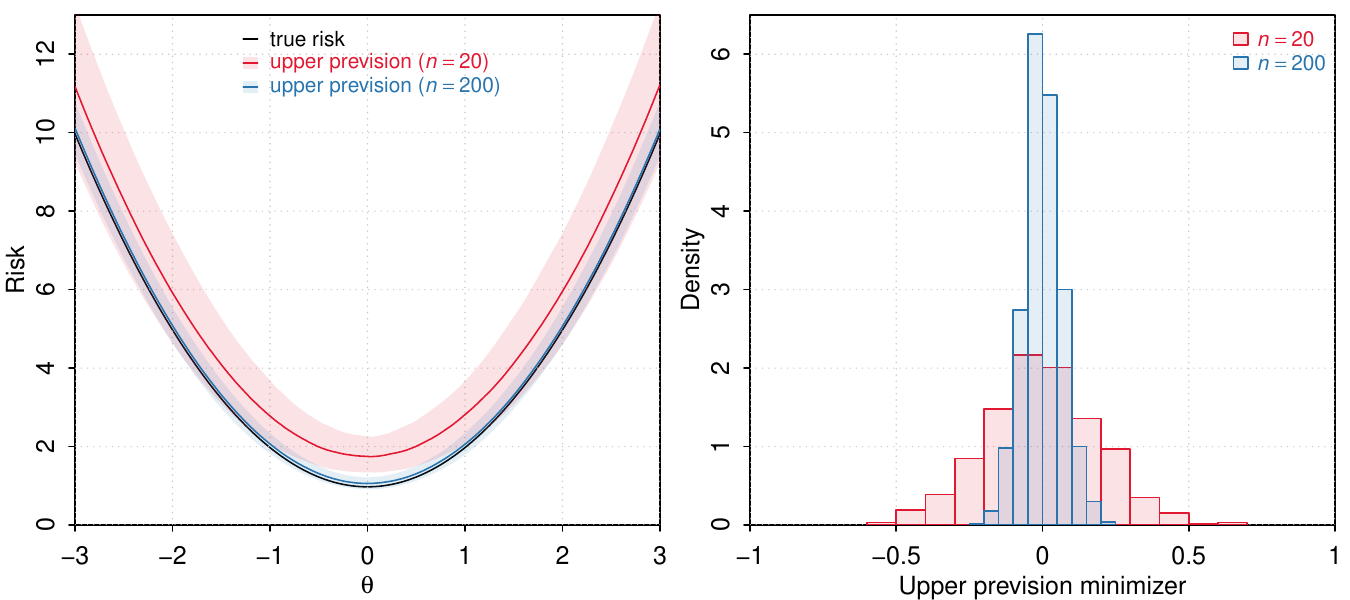}
  \caption{\footnotesize Summary of numerical illustration in the Dempster-Hill setup. Left: True risk (black) and upper prevision (red and blue) as functions of $\theta$. The solid, colored curves are median MFGF upper previsions pooled across 1000 replications at two sample sizes ($n = 20$ and 200), and the shaded areas indicate the middle 95\% upper prevision values across replications. Right: Histograms of MFGF upper prevision minimizers across 1000 replications.}
  \label{fig:sim}
\end{figure}

A simple numerical example is provided to demonstrate the recovery of the true risk by the MFGF upper risk in the Dempster-Hill setup.  Two sample size conditions were considered: $n = 20$ and 200.  We simulated data from a standard normal distribution truncated to the interval $[-3, 3]$ and focused on the squared error loss function $\ell(\vartheta, y) = (y - \vartheta)^2$.  A summary of simulation results can be found in Figure~\ref{fig:sim}.  It is easy to see that the true risk minimizer is zero, the mean of the truncated normal distribution.

 The median, the 5th, and the 95th percentiles of upper previsions across 1000 replications are displayed across a wide range of $\theta$ values in the left panel; as a benchmark, the true risk function is also superimposed. It is observed that the upper prevision overestimates the risk on average, with the discrepancy narrowing as the sample size grows. In the right panel of Figure~\ref{fig:sim}, we contrast the histograms of upper prevision minimizers obtained under the two sample size conditions. We note that the histograms of upper prevision minimizers center around the true risk minimizer (i.e., 0) and become more concentrated in larger samples.

\bibliographystyle{apalike}
\bibliography{references}

\end{document}